\documentclass[11pt]{amsart}
\usepackage[T1]{fontenc}
\usepackage{graphicx}
\usepackage{amsmath}
\usepackage{amssymb}
\usepackage{amsthm}
\usepackage{thmtools}
\usepackage{thm-restate}
\usepackage{hyperref}
\usepackage{cleveref}
\usepackage{geometry}
\usepackage{tikz}
\usepackage{tikz-cd}
\usepackage{todonotes}
\usepackage{stmaryrd}
\usepackage{cite}
\usepackage{booktabs}

\setuptodonotes{inline}
\usetikzlibrary{decorations.pathmorphing,decorations.pathreplacing,calligraphy, positioning}

\author{Anand Deopurkar}
\address{Anand Deopurkar, Mathematical Sciences Institute, Australian National University}
\email{anand.deopurkar@anu.edu.au}

\author{Anand Patel}
\address{Anand Patel, Department of Mathematics, Oklahoma State University}
\email{anand.patel@okstate.edu}

\date{\today}

\ifcsname theorem\endcsname{}\else\declaretheorem[parent=section]{theorem}\fi
\ifcsname corollary\endcsname{}\else\declaretheorem[sibling=theorem]{corollary}\fi
\ifcsname lemma\endcsname{}\else\declaretheorem[sibling=theorem]{lemma}\fi
\ifcsname proposition\endcsname{}\else\declaretheorem[sibling=theorem]{proposition}\fi
\ifcsname conjecture\endcsname{}\else\fi
\ifcsname problem\endcsname{}\else\fi
\ifcsname question\endcsname{}\else\fi
\ifcsname definition\endcsname{}\else\fi
\ifcsname exercise\endcsname{}\else\fi
\ifcsname example\endcsname{}\else\fi
\ifcsname remark\endcsname{}\else\declaretheorem[sibling=theorem, style=remark]{remark}\fi

\renewcommand {\P}{{\bf P}}
\providecommand {\Gr}{{\bf Gr}}
\providecommand {\A}{{\bf A}}

\providecommand{\GL}{\operatorname{GL}}
\providecommand{\PGL}{\operatorname{PGL}}

\providecommand {\from}{{\colon}}

\providecommand{\coker}{\operatorname{coker}}

\providecommand{\Hom}{\operatorname{Hom}}

\providecommand{\Sym}{\operatorname{Sym}}

\providecommand{\im}{\operatorname{im}}
\providecommand{\supp}{\operatorname{supp}}
 \renewcommand{\k}{\mathbf k}

\DeclareMathOperator{\Orb}{Orb}
\DeclareMathOperator{\WOrb}{WOrb}
\DeclareMathOperator{\Stab}{Stab}

\title{Orbits of linear series on the projective line}

\begin{document}

\maketitle

\begin{abstract}
  We compute the equivariant fundamental class of the orbit closure of a linear series on the projective line.
  We also describe the boundary of the orbit closure and how the orbits specialise in one parameter families.
\end{abstract}

\section{Introduction}
Let \(M\) be a variety with the action of an algebraic group \(G\).
Given a point \(m \in M\), let \(\Orb(m) \subset M\) be the closure of the \(G\)-orbit of \(m\).
The fundamental class of \(\Orb(m)\) in the equivariant Chow ring of \(M\) encodes fascinating geometric information.
The class is particularly meaningful when the \(G\)-orbits have geometric significance.
The main result of this paper (\Cref{thm:main}) is a complete and explicit expression for these classes for \(M = \Hom(\k^{r+1}, \Sym^d \k^2)\) with the natural action of \(G = \GL(r+1) \times \GL(2)\).
Here \(\k\) is an algebraically closed field in which \(d! \neq 0\).

The orbits of the \(\GL(r+1) \times \GL(2)\) action on \(\Hom(\k^{r+1}, \Sym^d \k^2)\) represent isomorphism classes of linear series of rank \(r\) and degree \(d\) on the projective line.
Their fundamental class turns out to depend only on the ramification profile.
For example, when \(r = 1\) and the linear series is base-point free, the class is a function of the multiplicities in the ramification divisor.
Likewise, when \(r = 2\) and the linear series is base-point free, the class depends on the multiplicities of cusps and hyperflexes.

The group \(\GL(r+1)\) acts freely on the open subset of \(\Hom(\k^{r+1}, \Sym^d\k^2)\) consisting of injective maps.
The quotient is the Grassmannian \(\Gr(r+1, \Sym^d\k^2)\).
So we have a surjective homomorphism
\begin{equation}\label{eq:homgrass}
  A^*_{\GL(r+1) \times \GL(2)}(\Hom(\k^{r+1}, \Sym^d\k^2)) \to A^*_{\GL(2)}(\Gr(r+1, \Sym^d\k^2)).
\end{equation}
Taking images under this map gives classes of \(\GL(2)\)-orbits in the Grassmannian.
Forgetting the group action gives a further homomorphism
\begin{equation}\label{eq:homnon}
  A^*_{\GL(2)}(\Gr(r+1, \Sym^d\k^2)) \to A^*(\Gr(r+1, \Sym^d\k^2)).
\end{equation}
Taking images under this map gives non-equivariant classes of the \(\GL(2)\)-orbits.
All of these, including the non-equivariant ones, were unknown.
\Cref{tab:quarticpencils} shows the weighted (multiplied by the order of the stabiliser) orbit classes for all base-point free pencils of quartics.
\begin{table}
  \begin{tabular}{p{.3\textwidth}p{.5\textwidth}} 
    \toprule
    Ramification profile & Weighted orbit class in the Schubert basis \\
    \midrule
   \(2,2,2,2,2,2\) & \(24 \cdot \Omega_3 + 48 \cdot \Omega_{2,1}\) \\
    \(3,2,2,2,2\) & \(16 \cdot \Omega_3 + 40 \cdot \Omega_{2,1}\) \\
    \(4,2,2,2\) & \(12 \cdot \Omega_3 + 24 \cdot \Omega_{2,1}\) \\
    \(3,3,2,2\) & \(8 \cdot \Omega_3 + 32 \cdot \Omega_{2,1}\)\\
    \(3,3,3\) & \(24 \cdot \Omega_{2,1}\) \\
    \(4,3,2\) & \(4 \cdot \Omega_3 + 16 \cdot \Omega_{2,1}\)\\
    \bottomrule
  \end{tabular}
  \caption{The weighted orbit classes of base-point free pencils of quartics in \(\Gr(2, \Sym^4\k^2)\) as a function of the ramification profile.}
  \label{tab:quarticpencils}
\end{table}

Many mathematicians have studied equivariant and non-equivariant orbit classes.
We were inspired mainly by the work of Aluffi and Faber \cite{alu.fab:00,alu.fab:93,alu.fab:93*1} on the (non-equivariant) classes of orbits of points on \(\P^{1}\) and curves in \(\P^{2}\).
There has been exciting recent progress on extending their work to the equivariant setting; see \cite{lee.pat.tse:19} for points on \(\P^{1}\) and quartic curves in \(\P^2\), \cite{lee.pat.spi.ea:20} for hyperplane arrangements in \(\P^n\), and \cite{deo.pat.tse:21} for generic cubics in \(\P^3\).
There is also significant work in a more representation theoretic context.
See, for example, \cite{feh.rim.web:20,buc.rim:07} for classes of orbits of quiver representations and \cite{ber.fin:17} for orbits of matrices.

\subsection{Main theorem}
To state the main result, we need some notation.
We have an identification
\begin{equation}\label{eq:A}
  A^{*}_{\GL(r+1) \times \GL(2)}(\Hom(\k^{r+1}, \Sym^d \k^2))= \mathbf{Z}[\mu_{0},\dots,\mu_{r}, \omega_{1}, \omega_2]^{S_{r+1} \times S_{2}},
\end{equation}
where the \(\mu_{i}\)'s and the \(\omega_{j}\)'s are the ``Chern roots'' of the tautological representations of \(\GL(r+1)\) and \(\GL(2)\), respectively.
Given integers \(a_{0}< \dots < a_{r}\) in \(\{0,\dots,d\}\), set 
\[ \phi_{a_0, \dots, a_r}(\mu;\omega) = \prod_{j = 0}^r\prod_{i}((d-i) \omega_1 + i \omega_2-\mu_{j}),\]
where the index \(i\) ranges over \(\{0,\dots,d\} - \{a_0, \dots, a_r\}\).
Set
\[ \psi_{a_0,\dots,a_r}(\mu;\omega) = \phi_{a_0, \dots, a_r}(\mu;\omega_1,\omega_2) - \phi_{a_0, \dots, a_r}(\mu;\omega_2,\omega_1).\]
Note that \(\psi\) is symmetric in the \(\mu\)'s and anti-symmetric in the \(\omega\)'s.

Given an injective \(s \in \Hom(\k^{r+1}, \Sym^d \k^2)\), let \(S \subset \Sym^{d}\k^2\) be its image.
Associated to \(S\) and a point \(p \in \P^{1}\), we have the vanishing sequence
\[a_{0}(p) < \dots < a_{r}(p)\]
characterised by the property that for each \(i\), the space \(S\) contains a section vanishing to order exactly \(a_{i}(p)\) at \(p\).
For all but finitely many points \(p \in \P^{1}\), the vanishing sequence is given by \(a_{i}(p) = i\).
Let \(B \subset \P^{1}\) be the points where it differs.
The ramification profile of \(S\) is the multi-set
\[\{(a_0(b), \dots, a_{r}(b)) \mid b \in B\}.\]
\begin{theorem}[\Cref{thm:realmain}]\label{thm:main}
  Let \(\k\) be an algebraically closed field in which \(d! \neq 0\).
  Let \(s \in \Hom(\k^{r+1}, \Sym^d  \k^2)\) be an injective map with image \(S \subset \Sym^d\k^2\).
  Let \(\{(a_{i}(b)) \mid b \in B\}\) be the ramification profile of \(S\).
  Let \(\Gamma \subset  \PGL(2)\) be the stabiliser of \(S \subset \Sym^d\k^2\), and assume that it is finite.
  Let \(\Orb(s)\) be the closure of the \(\GL(r+1) \times \GL(2)\) orbit of \(s\).
  Then the class
  \[ |\Gamma| \cdot [\Orb(s)] \in A^{(r+1)(d-r)-3}_{\GL(r+1) \times \GL(2)}(\Hom(\k^{r+1}, \Sym^d \k^2)) \]
  is given by
  \[|\Gamma| \cdot [\Orb(s)] = \frac{1}{(\omega_1-\omega_2)^3}\sum_{b \in B}   \psi_{a_0(b), \dots, a_r(b)}(\mu;\omega) + \frac{2-|B|}{(\omega_1-\omega_2)^3} \cdot \psi_{0,\dots,r}(\mu;\omega).\]
\end{theorem}
Despite the apparent denominators, the expression must be a polynomial (although this is not obvious).

To get the \(\GL(2)\)-equivariant class of \(\Orb(S) \subset \Gr(r+1, \Sym^d\k^2)\), we apply the homomorphism \eqref{eq:homgrass}, which sends the \(\mu\)'s to the Chern roots of the universal sub-bundle.
To get the non-equivariant class of \(\Orb(S) \subset \Gr(r+1, \Sym^d\k^2)\), we further substitute \(\omega_1 = \omega_2 = 0\).

\begin{remark}
  Although \Cref{thm:main} gives the non-equivariant classes by the series of substitutions mentioned above, it does not give a simple expression for it.
  Indeed, the substitution \(\omega_1 = \omega_2 = 0\) can only be performed after dividing by the denominator in \Cref{thm:main}, and the result of the division is not obvious.
\end{remark}
\begin{remark}
  Given \([S] \in \Gr(r+1, \Sym^d W)\), we can consider \([S]\) as a point of the projective space \(\P \left(\wedge^{r+1} \Sym^d W\right)\) by applying the Plucker embedding.
  The class of \(\Orb([S])\) in \(\P \left(\wedge^{r+1} \Sym^d W\right)\) depends only on the multiplicities in the ramification divisor of \(S\).
  The class of \(\Orb([S])\) in \(\Gr(r+1, \Sym^d W)\) is a richer invariant, sensitive to the full ramification profile.
\end{remark}

\subsection{Orbit closures}
The method of proof of the main theorem also yields a description of the boundary points of the orbit closure.
\begin{theorem}[\Cref{thm:orbitclosure}]\label{thm:closure}
  The orbit closure \(\Orb(S) \subset \Gr(r+1, \Sym^d\k^2)\) is a union of finitely many orbits.
  In addition to the orbit of \(S\), these are the orbits of subspaces of the form
  \[\langle w_1^{d-a_0}w_2^{a_0}, \dots, w_1^{d-a_r}w_2^{a_r} \rangle,\]
  where \(a_0 < \dots < a_r\) is the vanishing sequence of \(S\) at some point of \(\P^1\) and \(w_1, w_2\) is a basis of \(\k^2\).
\end{theorem}
At a generic point of \(\P^1\), the vanishing sequence is given by \(a_i = i\).
The orbit of the corresponding linear series \(\langle  w_1^{d}w_2^0, \dots, w_1^{d-r}w_2^r \rangle\) is \(1\)-dimensional.
At a ramification point, the vanishing sequence is different, and the orbit of the corresponding linear series \(\langle  w_1^{d-a_0}w_2^{a_0}, \dots, w_1^{d-a_r}w_2^{a_r} \rangle\) is \(2\)-dimensional.

\subsection{Orbit specialisation}
Consider a one-parameter family of linear series in which some ramification points collide.
Then the flat limit of the orbit closure contains the orbit closure of the limiting linear series.
In a suitable family, the limit of the orbit closure contains the orbit closure of a unique other linear series, which we now describe.
Let \(A\) be the multi-set of vanishing sequences of the points that collide and \(B\) the multi-set of  vanishing sequences of the remaining points.
Let \(c = (c_0, \dots, c_r)\) be the vanishing sequence of the limiting linear series at the point of collision.
Set \(c' = (d-c_r, \dots, d-c_0)\).
Denote by \(\WOrb\) the weighted orbit closure, namely
\[ \WOrb = |\Gamma| \cdot \Orb.\]
Recall that we can treat the class \([\WOrb]\) as a function of the ramification profile.
\begin{theorem}[\Cref{prop:introspec}]\label{thm:introspec}
  In the notation above,  we have the equality of \(\GL(r+1) \times \GL(2)\)-equivariant classes
  \[ [\WOrb(A \cup B)] = [\WOrb(A \cup \{c'\})] + [\WOrb(B \cup \{c\})].\]
\end{theorem}
The statement above is a special case of a large family of linear relations among the classes \([\WOrb]\); see \Cref{sec:spec} for more.

By repeated collisions of ramification points, we can express an arbitrary orbit class as a non-negative linear combination of orbit classes of linear series with 3 ramification points (the expression is not unique).
For example, in the case of quartic pencils, \Cref{fig:quarticspec} shows a sequence of collisions and the resulting complementary orbits.
As a result, we get the following equality
\[
  \begin{split}
    \WOrb((0,2) \times 6) &= \\
    \WOrb((0,3)&\times 2, (1,2)) + 2 \WOrb((0,2) \times 2, (1,4)) + \WOrb((0,2) \times 2, (2,3)).
  \end{split}
\]
In the non-equivariant Chow ring of the Grassmannian, the above evaluates to the identity
\[
  24 \cdot \Omega_3 + 48 \cdot \Omega_{2,1} =
    \left(8 \Omega_3 + 20 \cdot \Omega_{2,1}\right) + 2 \cdot \left(8 \cdot \Omega_{3} + 8 \cdot \Omega_{2,1} \right) + 12 \cdot \Omega_{2,1}.
\]

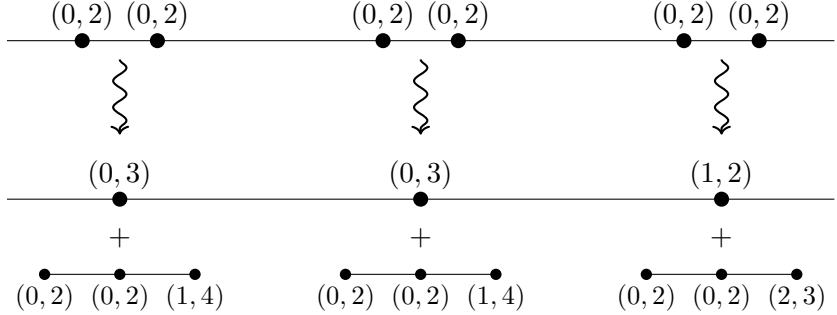
\begin{figure}
\begin{tikzpicture}
  \draw (-6,0) -- (5,0);
  \fill
  (-5,0) circle (.1) node[above] {\((0,2)\)}
  (-4,0) circle (0.1) node[above] {\((0,2)\)}
  (-1,0) circle (0.1) node[above] {\((0,2)\)}
  (0,0) circle (0.1) node[above] {\((0,2)\)}
  (3,0) circle (0.1) node[above] {\((0,2)\)}
  (4,0) circle (0.1) node[above] {\((0,2)\)};

  \begin{scope}[yshift = -60]
  \draw (-6,0) -- (5,0);
  \fill
  (-4.5,0) circle (.1) node[above] {\((0,3)\)} node[below=0.25cm] {\(+\)}
  (-0.5,0) circle (0.1) node[above] {\((0,3)\)} node[below=0.25cm] {\(+\)}
  (3.5,0) circle (0.1) node[above] {\((1,2)\)} node[below=0.25cm] {\(+\)};

  \draw(-5.5,-1) -- (-3.5,-1);
  \draw(-1.5,-1) -- (0.5,-1);
  \draw(2.5,-1) -- (4.5,-1);
  \fill (-5.5,-1) circle (0.075) node[below] {\small \((0,2)\)};
  \fill (-4.5,-1) circle (0.075) node[below] {\small \((0,2)\)};
  \fill (-3.5,-1) circle (0.075) node[below] {\small \((1,4)\)};

  \fill (-1.5,-1) circle (0.075) node[below] {\small \((0,2)\)};
  \fill (-0.5,-1) circle (0.075) node[below] {\small \((0,2)\)};
  \fill (0.5,-1) circle (0.075) node[below] {\small \((1,4)\)};

  \fill (2.5,-1) circle (0.075) node[below] {\small \((0,2)\)};
  \fill (3.5,-1) circle (0.075) node[below] {\small \((0,2)\)};
  \fill (4.5,-1) circle (0.075) node[below] {\small \((2,3)\)};
\end{scope}

\draw[thick, decoration=snake, ->, decorate] (-4.5,-0.25) to (-4.5, -1.25);
\draw[thick, decoration=snake, ->, decorate] (-0.5,-0.25) to (-0.5, -1.25);
\draw[thick, decoration=snake, ->, decorate] (3.5,-0.25) to (3.5, -1.25);
\end{tikzpicture}
\caption{An example of orbit specialisation for a generic quartic pencil. Collisions of ramification points create a new ramification profile and a unique complementary orbit shown underneath.}
\label{fig:quarticspec}
\end{figure}

\subsection{Ideas in the proof}
To find the class of the closure of the \(G\)-orbit of \(m \in M\), we use a complete orbit parametrisation for \(m\).
A complete orbit parametrisation consists of a proper \(G\)-variety \(X\) with a dense orbit and an equivariant map \(\phi \colon X \to M\) whose image contains \(m\).
Then, up to a scalar, the class we seek is \(\phi_{*}[X]\).

We begin with \(\P^3 = \P\Hom(\k^{2}, \k^2)\) with the natural \(\GL(2)\) action on the target \(\k^2\) and the equivariant rational map
\[\phi \colon \P^3 \dashrightarrow \Gr(r+1, \Sym^d\k^2)\]
defined by \(m \mapsto m(S)\).
The map is only rational because when \(m\) is degenerate, the image \(m(S) \subset \Sym^d\k^2\) need not be \((r+1)\)-dimensional.
We prove (\Cref{thm:res}) that \(\phi\) extends to a regular map on a simple blow up:
\begin{equation}\label{eq:res}
  \widetilde{\P}^3 \to \Gr(r+1, \Sym^d\k^2).
\end{equation}
This is the same blow-up that Aluffi and Faber use in \cite{alu.fab:93*1} for \(r = 0\).
We compute the push-forward using equivariant localisation as in \cite[Appendix~B]{lee.pat.tse:19}.
We upgrade from the orbit class in \(\Gr(r+1,\Sym^d\k^2)\) to the orbit class in \(\Hom(\k^{r+1}, \Sym^d\k^2)\) using general principles, outlined in \Cref{sec:grass}.

\subsection{Organisation}
In \Cref{sec:grass}, we discuss equivariant orbit classes in general and classes on Grassmannians in particular.
We also discuss general results about one parameter specialisations of orbits.
In \Cref{sec:van}, we recall vanishing sequences, the Wronskian, and the ramification divisor of a linear series.
In \Cref{sec:res}, we begin the proof of the main theorem by constructing a complete parametrisation for a linear series.
In \Cref{sec:localisation}, we finish the proof of the main theorem with a localisation computation.
In \Cref{sec:spec}, we discuss one-parameter specialisations of orbits of linear series, and the resulting additive decomposition of the orbit class.

\subsection{Conventions}
Throughout, \(\k\) denotes an algebraically closed field of characteristic 0 or of characteristic \(> d\).
All schemes are of finite type over \(\k\).
A \emph{variety} is a reduced scheme, separated over \(\k\).
Given a vector space \(V\), the notation \(\underline V\) denotes the trivial vector bundle with fiber \(V\) on the scheme that is clear from the context.
Given a vector bundle \(V\), the \emph{projectivisation} \(\P V\) is the space of one-dimensional sub-bundles.
Unless mentioned otherwise, \emph{point} means a closed point.
Given a variety \(X\) with a \(G\)-action, \(A^i_G(X)\) denotes the \(G\)-equivariant Chow group of \(X\) of cycles of co-dimension \(i\), as defined in \cite{edi.gra:98}.

\subsection{Acknowledgements}
We started working on the project when Anand P. visited Anand D. at the Institute for Computational and Experimental Research in Mathematics (ICERM) at Brown University during the semester program \emph{Braids} in 2022.
We thank ICERM and the organisers of \emph{Braids}.
Anand D. thanks the Australian Research Council for the grant \texttt{DE180101360} that supported a part of this project.

\section{Equivariant orbit classes}\label{sec:grass}
Let \(G\) be an algebraic group and \(M\) a variety with a \(G\) action.
A \emph{complete parametrisation} of the \(G\) orbit of \(m \in M\) is a \(G\)-variety \(X\) together with a proper \(G\)-equivariant map \(\phi \colon X \to M\) such that
\begin{enumerate}
\item there exists an \(x \in X\) such that \(\phi (x) = m\),
\item the  \(G\)-orbit of \(x\) is dense in \(X\), and
\item the stabiliser of \(x \in X\) is trivial.
\end{enumerate}

Let \(\Gamma \subset G\) be the stabiliser of \(m \in M\) and let \(\phi \colon X \to M\) be a complete parametrisation of the \(G\) orbit of \(m\).
Then we have
\begin{equation}\label{eq:worb}
  \phi_*[X] = \begin{cases}
    |\Gamma| \Orb(m) & \text{ if \(\Gamma\) is finite} \\
    0 & \text{otherwise}.
    \end{cases}
\end{equation}

We now specialise to two particular settings.

In the first setting, we take \(M = W\), where \(W\) is a non-zero \(G\)-representation and the closely related setting of \(M = \P W\).
Assume that the image of \(G \to \GL(W)\) contains the scalar matrices.
Suppose the map \(G \to \PGL(W)\) factors through \(G \to \overline G\).
(In applications, \(G\) will be a general linear group and \(\overline G\) will be the corresponding projective general linear group.)

Let
\[ 0 \to O(-1) \to \underline W \to Q \to 0\]
be the universal sequence on \(\P W\).
\begin{proposition}\label{prop:intzeta}
  Let \(w \in W\) be non-zero.
  Let \(\Gamma \subset \overline G\) be the stabiliser of \([w] \in \P W\) and assume that it is finite.
  Let \(\phi \colon X \to \P W\) be a complete parametrisation of the \(\overline G\) orbit of \([w] \in \P W\).
  Then, in \(A^*_G(W) = A^*_G(\bullet)\), we have
  \[ |\Gamma| [\Orb(w)] = \int_X c_{\rm top}(\phi^* Q).\]
\end{proposition}
\begin{proof}
  Let \(\pi \colon W^* \to \P W\) be the projection.
  Since the image of \(G\) contains the scalars, we have
  \[ \Orb(w) = \pi^* \Orb([w]),\]
  and hence
  \begin{equation}\label{eq:push2}
    |\Gamma| \cdot \Orb(w) = |\Gamma| \cdot \pi^* \Orb([w]).
  \end{equation}
  Since \(\phi\) is a complete parametrisation of the \(\overline G\) orbit of \([w]\), we have
  \[ |\Gamma| \cdot \Orb([w]) = \phi_*[X].\]
  We now compute the right hand side.
  
  On \(X \times \P W\), let \(\pi_i\) for \(i = 1,2\) be the two projections.
  Consider the composite
  \[ \pi_2^*O(-1) \to \underline W \to \pi_1^*\phi^*Q.\]
  Its vanishing locus is precisely the graph of \(\phi \colon X \to \P W\).
  Therefore, we have
  \begin{equation}\label{eq:push1}
    \phi_*[X] = {\pi_2}_* c_{\rm top} \Hom(\pi_2^*O(-1), \pi_1^* \phi^*Q).
  \end{equation}

  In \eqref{eq:push2}, we substitute \(\phi_*[X]\) from \eqref{eq:push1} and use the push-pull formula on the fiber square
  \[
    \begin{tikzcd}
      X \times W^* \ar{r}\ar{d}& W^* \ar{d}{\pi}\\
      X \times \P W \ar{r}& \P W.
    \end{tikzcd}
  \]
  Let \(\widetilde \pi_i\) for \(i = 1,2\) be the two projections on \(X \times W^*\).
  Noting that the pull-back of \(O(-1)\) to \(W^*\) is trivial, in \(A^*_G(W^*)\) we have
  \begin{equation}\label{eq:push3}
    \begin{split}
      |\Gamma|[\Orb(w)] &= {\widetilde {\pi_2}}_* c_{\rm top} \Hom(O, \widetilde \pi_1^* \phi^*Q)\\
      &= \int_X c_{\rm top}(\phi^*Q).
    \end{split}
  \end{equation}
  Since \(A^*_G(W^*) = A^*_G(W)\) in co-dimension less than \(\dim W\), equation \eqref{eq:push3} also holds in \(A^*_G(W)\).
\end{proof}

In the next setting, we assume that \(G = \GL(V) \times H\) and \(M = \Hom(V, W)\), where \(W\) is an \(H\)-representation.
Let \(\dim V = r+1\).
A closely related setting is that of \(M = \Gr(r+1, W)\) with the induced action of \(H\).
Suppose the map \(H \to \PGL(W)\) factors through \(H \to \overline H\).
(In applications, \(H\) will be a general linear group and \(\overline H\) will be the corresponding projective general linear group.)

Let
\[ 0 \to E \to \underline W \to Q \to 0\]
be the universal sequence on the Grassmannian.
\begin{proposition}\label{prop:intgr}
  Let \(s \in \Hom(V, W)\) be an injective homomorphism with image \(S\).
  Let \(\Gamma \subset \overline H\) be the stabiliser of \([S] \in \Gr(r+1, W)\) and assume that it is finite.
  Let \(\phi \colon X \to \Gr(r+1, W)\) be a complete parametrisation of the \(\overline H\)-orbit of \([S]\).
  Then, in \(A^*_G(\Hom(V,W)) = A^*_G(\bullet)\), we have
  \[ |\Gamma|[\Orb(s)] = \int_X c_{\rm top} \Hom(\underline V, \phi^*Q).\]
\end{proposition}
\begin{proof}
  Consider the \(H\)-equivariant bundle \(\phi^* E\) on \(X\).
  Set
  \[F = \Hom(\underline V, \phi^*E).\]
  Then \(F\) admits an action of \(G = \GL V \times H\).
  Set \(Y = \P F\); let \(\pi \colon Y \to X\) be the projection; and let \(\psi = \pi \circ \phi\).
  The composite
  \begin{equation}\label{eq:twosteps}
    O_Y(-1) \xrightarrow{i} \Hom(\underline V , \psi^*E) \xrightarrow{j} \Hom(\underline V , \underline W)
    \end{equation}
  yields a \(G\)-equivariant map
  \[ Y \to \P \Hom(V, W).\]
  This map is a complete parametrisation of the \(\PGL V \times \overline H\)-orbit of \([s]\).
  It is easy to check that the stabiliser of \([s]\) in \(\PGL V \times \overline H\) maps isomorphically to \(\Gamma \subset \overline H\) under the second projection.
  By \Cref{prop:intzeta}, in \(A^*_G(\Hom(V,W))\) we have
  \[ |\Gamma|[\Orb(s)] = \int_Y c_{\rm top}(\coker(j \circ i)).\]
  From the sequence in \eqref{eq:twosteps}, we have an exact sequence
  \[ 0 \to \coker i \to \coker(j \circ i) \to \coker j \to 0.\]
  and hence the equality of top Chern classes
  \[ c_{\rm top}(\coker j \circ i) = c_{\rm top}(\coker i) \cdot c_{\rm top}(\coker j).\]
  Observe that
  \[ \coker j = \pi^* \Hom(\underline V, \phi^* Q). \]
  Since \(Y \to X\) is a projective space bundle and \(i\) is its universal sub-bundle, we have
  \[ \int_{Y/X} c_{\rm top}(\coker i) = [X].\]
  Now, the push-pull formula gives
  \[
    \int_{Y/X} c_{\rm top}(\coker i) \cdot c_{\rm top}(\coker j) = c_{\rm top}(\Hom(\underline V, \phi^* Q)).
  \]
  Integrating the two sides along \(X\) gives the result.
\end{proof}

We now study specialisations of orbits in families.
The results are essentially from \cite[\S~2.5]{lee.pat.tse:19}, but we state and prove them more generally.

Fix an action of an algebraic group \(G\) on a variety \(M\).
For simplicity, assume that \(G\) is irreducible.
Given \(m \in M\), define the \(G\)-equivariant cycle
\[
  \WOrb(m) = \begin{cases}
    |\Stab(m)| \cdot \Orb(m) & \text{ if \(\Stab(m)\) is finite,} \\
    0 & \text{otherwise}.
  \end{cases}
\]
Let \(\Delta\) be the spectrum of a DVR with general point \(\eta\) and special point \(0\).
Fix a section \(m_\eta \from \eta \to M_\eta\).
Let \(\Sigma\) be the cycle on \(M\) obtained as the limit of the cycle \(\WOrb(m_\eta)\) on \(M_\eta\).
Our goal is to understand \(\Sigma\).
We begin with an easy observation.

\begin{proposition}\label{prop:onelimit}
  Suppose \(m_\eta\) has a limit \(\overline m \in M\).
  Then the cycle \(\Sigma - \WOrb(\overline m)\) is effective.
\end{proposition}
\begin{proof}
  Let \(X \supset G\) be an equivariant compactification.
  Consider the map \(\alpha \colon G_\eta \to M_\eta\) defined by \(g \mapsto g m_\eta\), and let \(Y \subset X \times M \times \Delta\) be the closure of the graph of \(\alpha\).
  Then \(Y\) is irreducible of dimension \(\dim G + 1\).
  Let \(\pi \colon Y \to M\) be the projection.
  Observe that we have
  \[ {\pi}_* Y_\eta = \WOrb(m_\eta) \text{ and } {\pi}_* Y_0 = \Sigma.\]
  By construction, \(Y\) is invariant under the \(G\)-action on \(X \times M \times \Delta\).
  Since \(y = (1,\overline m, 0) \in Y_0\), we have a copy of \(G \subset Y_0\) given by \(G \cdot y\).
  Let \(Z \subset Y_0\) be its closure.
  Note that \({\pi}_* Z = \WOrb(\overline m)\).
  Furthermore, since \(\dim Z = \dim Y_0\), we see that \(Z \subset Y_0\) is an irreducible component.
  Then \(\Sigma - \WOrb(\overline m)\) is the push-forward of the effective cycle \([Y_0] - [Z]\).
\end{proof}

We have the following strengthening of \Cref{prop:onelimit} (see \cite[Proposition~2.13]{lee.pat.tse:19}).
\begin{proposition}\label{prop:manylimits}
  Consider sections \(g_1, \dots, g_n \from \eta \to G_\eta\) such that for all \(i \neq j\), the section \(g_ig_j^{-1}\) does not extend to a section \(\Delta \to G_\Delta\).
  Suppose \(\overline m_1, \dots, \overline m_n \in M\) are limits of \(g_1 m_\eta, \dots, g_n m_\eta\).
  Then the cycle \(\Sigma - \WOrb(\overline m_1) - \dots - \WOrb(\overline m_n)\) is effective.
\end{proposition}
\begin{proof}
  Let \(X \supset G\) be an equivariant compactification.
  Consider the map \(\alpha \colon G_\eta \to G_\eta^n \times M_\eta\) given by
  \[ \alpha \colon g \mapsto (gg_1^{-1}, \dots, gg_n^{-1}, gm_\eta)\]
  and let \(Y \subset X^n \times M \times \Delta\) be the closure of its image.
  Denoting by \(\pi\) the projection onto the \(M\) factor, we have
  \[ {\pi}_* Y_\eta = \WOrb(m_\eta) \text{ and } {\pi}_* Y_0 = \Sigma.\]
  Let \(\overline g_{i,j} \in X\) be the limit of \(g_ig_j^{-1}\).
  Then \(\overline g_{i,j} \not \in G\) for \(i \neq j\).
  Then the limit of \(\alpha(g_i)\) is the point
  \[ y_i = (\overline g_{i,1}, \dots, \overline g_{i,i-1}, 1, \overline g_{i,i+1}, \overline g_{i,n}, \overline m_i) \in Y_0.\]
  Each \(y_i\) gives a copy of \(G \subset Y_0\), say \(G_i\), defined by \(G_i = G \cdot y_i\).
  Since \(\overline g_{i,j} \not \in G\) for \(i \neq j\), the \(j\)-th components of the points of \(G_i\) do not lie in \(G\) for \(i \neq j\) and lie in \(G\) for \(i = j\).
  In particular, for \(i \neq j\), the \(G_i\) and \(G_j\) are disjoint.  
  Let \(Z_i\) be the closure of \(G_i\).
  Then the \(Z_i\) are distinct irreducible components of \(Y_0\).
  The cycle \(\Sigma - \WOrb(\overline m_1) - \dots - \WOrb(\overline m_n)\) is the \(\pi\)-pushforward of the effective cycle \(Y_0 - Z_1 - \dots - Z_n\).
\end{proof}

\section{Vanishing sequences and the Wronskian}\label{sec:van}
Let \(C\) be a smooth curve, \(L\) a line bundle of degree \(d\) on \(C\), and \(S \subset H^0(C,L)\) a subspace of dimension \(r+1\).
The pair \((S,L)\) is called a \emph{linear series of rank \(r\) and degree \(d\)} or a \(g^r_d\) on \(C\).
Given a point \(p \in C\), there exists a unique sequence of integers \(a_0(p) < \dots < a_r(p)\) in \(\{0,\dots,d\}\) such that \(S\) contains a section \(s_i\) that vanishes to order exactly \(a_i(p)\) at \(p\).
The sequence \(a_i(p)\) is called the \emph{vanishing sequence} of \(S\) at \(p\).
The sections \(s_0, \dots, s_r\) form a basis of \(S\).

Denote by \(j_{r+1}L\) the bundle of jets of order \(r+1\) of \(L\);
that is 
\[ j_{r+1}L = {\pi_2}_* \left( \pi_1^* L \otimes O_{(r+1) \Delta} \right).\]
Here \(\Delta \to C \times C\) is the diagonal and the \(\pi_i\) are the two projections.
The linear series \(S \subset H^0(C, L)\) induces an evaluation map
\[ \underline S \to L,\]
and also a map
\begin{equation}\label{eq:jetmap}
  \underline S \to j_{r+1} L.
\end{equation}
The \emph{Wronskian} of \(S\) is the determinant of \eqref{eq:jetmap}.

\begin{proposition}\label{prop:wronskian}
  The Wronskian is non-zero and vanishes to order \(\sum a_i(p)-i\) at \(p\).
\end{proposition}
\begin{proof}
  The proof is a part of \cite[Proposition~1.1]{eis.har:83}.
  We give it for completeness.
  
  Let \(s_i \in S\) vanish to order \(a_i = a_i(p)\) at \(p\).
  Denote by \(\widehat M\) the completion of an \(O_C\)-module \(M\) at \(p\).
  Choose an isomorphism of \(\k\llbracket t \rrbracket\)-modules \(\widehat L \cong \k\llbracket t \rrbracket\).
  Then we have an isomorphism 
  \[
    \widehat{j_{r+1}L} \cong \k\llbracket t,u \rrbracket/(u-t)^{r+1}.
  \]
  If the section \(s \in \widehat L \cong \k\llbracket t\rrbracket\) is given by the power series \(s(t)\), then the induced section \(j_{r+1}(s)\) of \(\widehat{j_{r+1} L}\) is given by
  \[ j_{r+1}(s) = s(t) + s'(t)(u-t) + \dots + s^{(r)}(t)\frac{1}{r!} (u-t)^r.\]
  Let \(s_i\) correspond to the power series \(s_i(t) = t^{a_i}z_i \in \k\llbracket t \rrbracket\), where \(z_i \in \k\llbracket t \rrbracket\) is a unit.
  With respect to the bases \(s_0, \dots, s_{r+1}\) of \(S\) and \(1, (u-t), \dots, (u-t)^{r}/r!\) of \(\widehat{j_{r+1}L}\), the matrix of the map \eqref{eq:jetmap} is given by
  \[
    \begin{pmatrix}
      t^{a_0} z_0 & t^{a_1}z_1 & \dots & t^{a_r}z_r \\
      (t^{a_0} z_0)' & (t^{a_1}z_1)' & \dots & (t^{a_r}z_r)' \\
      \vdots & \vdots & \ddots & \vdots \\
      (t^{a_0} z_0)^{(r)} & (t^{a_1}z_1)^{(r)} & \dots & (t^{a_r}z_r)^{(r)}.
    \end{pmatrix}
  \]
  The lowest degree term in its determinant is \(t^{\sum a_i - i}\) and its coefficient is the determinant of the \((r+1) \times (r+1)\) matrix \(C\) with \(C_{ij} = a_i(a_i-1)\cdots(a_i-j+1)\).
  It remains to show that the coefficient is non-zero.
  By a sequence of invertible row transformations, the matrix \(C\) reduces to the Vandermonde matrix \(V\) with \(V_{ij} = a_i^j\).
  Since the \(a_i\)'s are distinct integers in \(\{0,\dots, d\}\), their images are distinct in \(\k\), which has characteristic 0 or characteristic \(> d\).
  Hence the determinant of \(V\), and hence of \(C\), is non-zero.
\end{proof}

Let \(\rho(S) \subset C\) be the vanishing locus of the Wronskian of \(S\).
We call \(\rho(S)\) the \emph{ramification divisor} of \(S\) and its points the \emph{ramification points}.
\begin{corollary}
  For all \(p \in C - \rho(S)\), we have \(a_i(p) = i\).
  In particular, for all but finitely many \(p\), we have \(a_i(p) = i\).
\end{corollary}
\begin{proof}
  Follows immediately from \Cref{prop:wronskian}.
\end{proof}

Note that \(\rho(S)\) is the zero-locus of a section of \(\det j_{r+1} L\).
Recall that we have a sequence of surjective maps
\[ j_{r+1}L \to j_{r}L \to \dots \to j_2 L \to j_1 L = L \to 0\]
with
\[ \ker (j_{i+1}L \to j_iL) = L \otimes \Omega_C^{i}.\]
As a result, we get
\begin{equation}\label{eq:jetdet}
  \det j_{r+1} L  = L^{r+1} \otimes \Omega_C^{r(r+1)/2}.
\end{equation}

Consider the Grassmannian \(\Gr = \Gr(r+1, H^0(C,L))\) with the universal sub-bundle
\[ S \subset \underline{H^0(C,L)}.\]
On \(\Gr \times C\), we have the universal evaluation map
\[ \pi_1^* S \to \pi_2^* L,\]
which induces a map
\[ \pi_1^* S \to \pi_2^* j_{r+1} L.\]
Taking the determinant of the last map and applying \({\pi_1}_*\) gives the following map on \(\Gr\):
\[ \det S \to \underline {H^0(C, \det j_{r+1} L)}.\]
By \Cref{prop:wronskian}, the map above is non-zero at every point of \(\Gr\), and hence yields a morphism
\begin{equation}\label{eq:wronskimap}
  \rho \colon \Gr(r+1, H^0(C, L)) \to \P H^0(C, \det j_{r+1}L).
\end{equation}
\begin{proposition}\label{prop:wronskimap}
  The map \(\rho\) in \eqref{eq:wronskimap} is finite.
  For \(C = \P^1\), it is also surjective.
\end{proposition}
\begin{proof}
  Assume that \(\dim H^0(C, L) > r+1\); otherwise, the statement is vacuous.
  
  Since the Grassmannian has Picard rank one, any morphism from it is either finite or constant.
  It is easy to see that \(\rho\) is not constant; so it must be finite.

  For \(C = \P^1\) and \(L = O(d)\), we have by \eqref{eq:jetdet}
  \[ \det j_{r+1} L \cong O((r+1)d) \otimes O(-r(r+1)) = O((r+1)(d-r)).\]
  Hence, both the source and target of \(\rho\) have the same dimension, namely \((r+1)(d-r)\).
\end{proof}
  
\section{Complete parametrisation of the orbit of a linear series}\label{sec:res}
Fix a \(2\)-dimensional vector space \(V\) and an \((r+1)\)-dimensional subspace \(S \subset \Sym^{d}V\).

On \(\P V \cong \P^1\), let \(O(-1) \subset \underline V\) be the universal sub-bundle and \(Q\) the quotient \(\underline V / O(-1)\).
Then we have a canonical identification \(\Sym^d V = H^0(\P V, Q^d)\).
We regard \(S\) as a linear series associated to the line bundle \(Q^d\).

Let \(W\) be another \(2\)-dimensional vector space.
Consider the 
\[ \P^{3} = \P \Hom(V,W).\]
We have a rational map
\[ \phi \colon \mathbf{P}^3 \dashrightarrow \Gr(r+1, \Sym^{d}W)\]
defined by
\begin{equation}\label{eq:mmap}
  m \mapsto [m(S)].
\end{equation}
Note that the map is equivariant with respect to the obvious \(\GL W\) action.
The goal of this section is to understand a \(\GL W\)-equivariant resolution of \(\phi\).

Let \(B \subset \mathbf{P} V\) be the ramification divisor of \(S\).
For every \(b \in B\), let \(L_{b} \subset \mathbf{P}^{3}\) be the linear subspace
\begin{equation}\label{eq:lb}
  L_{b} = \{[m] \mid m(b) = 0\}.
\end{equation}
Then \(L_{b}\) is a copy of  \(\P^{1}\) linearly embedded in \(\P^3\), and for \(b_{1} \neq b_{2}\), the lines \(L_{b_1}\) and \(L_{b_2}\) are disjoint.
Let
\[ \beta \colon \widetilde {\P}^{3} \to \P^{3}\]
be the blow-up at \(\bigsqcup_{b \in B} L_{b}\).
Note that the lines \(L_{b}\) are \(\GL W\)-invariant and hence the blow-up \(\widetilde{\P}^{3}\) inherits an action by \(\GL W\).
See \Cref{fig:blowup} for a sketch of \(\P^3\) and the lines \(L_b\).

\begin{figure}
  \centering
  \begin{tikzpicture}
    \draw[thick, red]
    (-1, -1.7) -- (0,1.6)
    (0, -1.7) -- (1,1.6);
    \draw[thick, blue, dashed]
    (0.5,-1.7) -- (-1,1.6);
    \draw[fill=black!5!white] (0,1.5) ellipse (2cm and 0.25cm);
    \draw (2,-1.5) arc (0:-180:2cm and 0.25cm);
    \draw plot [smooth, tension=1] coordinates {(-2,1.5) (-1,0) (-2,-1.5)};
    \draw plot [smooth, tension=1] coordinates {(2,1.5) (1,0) (2,-1.5)};
  \end{tikzpicture}
  \caption{
    The map \(\phi \colon \P^3 \dashrightarrow \Gr(r+1,\Sym^d W)\) becomes regular after blowing up the lines \(L_b\) (red) on the discriminant quadric \(\Delta\).
    Along the lines of the opposite ruling (blue, dashed), the map is constant.    
  }\label{fig:blowup}
\end{figure}
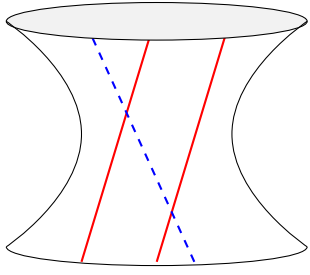

The key observation is the  following consequence of a result of Aluffi and Faber \cite{alu.fab:93*1}.
\begin{theorem}\label{thm:res}
  The rational map
  \[\phi \colon \P^{3} \dashrightarrow \Gr(r+1, \Sym^{d}W)\] extends to a regular map
  \[\widetilde \phi \colon \widetilde{\P}^{3} \to \Gr(r+1, \Sym^{d}W).\]
\end{theorem}
For the proof, we need a simple lemma.
\begin{lemma}\label{lem:zmt}
  Let \(\rho \colon X \to Y\) be a finite map between varieties.
  Let \(P\) be an irreducible normal variety with a rational map \(P \dashrightarrow X\).
  If the composite \(P \dashrightarrow Y\) extends to a regular map \(P \to Y\), then so does \(P \dashrightarrow X\).
\end{lemma}
\begin{proof}
  Let \(\Gamma_X \subset P \times X\) and \(\Gamma_Y \subset P \times Y\) be the closures of the graphs of \(P \dashrightarrow X\) and \(P \dashrightarrow Y\), respectively.
  Then the projections \(\Gamma_X \to P\) and \(\Gamma_Y \to P\) are birational.
  In fact, by the hypothesis on \(P \dashrightarrow Y\), the projection \(\Gamma_Y \to P\) is an isomorphism.
  The map \(\Gamma_X \to \Gamma_Y\) induced by \(\rho\) is finite and birational.
  Since the target \(\Gamma_Y \cong P\) is normal, we conclude that \(\Gamma_X \to \Gamma_Y\) is an isomorphism.
  Hence the projection \(\Gamma_X \to P\) is also an isomorphism.
  We deduce that \(P \dashrightarrow X\) extends to a regular map.
\end{proof}

\begin{proof}[Proof of \Cref{thm:res}]
  Set
  \[ \Gr = \Gr(r+1, \Sym^d W)\]
  and
  \[ \P = \P H^0(\P^1, \det j_{r+1} O(d)) = \P \Sym^{(r+1)(d-r)}W.\]
  Recall that we have the finite map
  \[ \rho \colon \Gr \to \P\]
  defined by the Wronskian.
  By definition, \(\rho\) is \(\GL W\)-equivariant.

  Let \(R\) be the Wronksian of \(S\).
  Consider the rational map
  \[ \P^3 \dashrightarrow \P\]
  defined by
  \[ m \mapsto [m(R)].\]
  Then we have the commutative diagram
  \[
    \begin{tikzcd}
      \widetilde \P^3 \ar[dashed]{r} \ar[equal]{d}& \Gr \ar{d}{\rho}  \\
      \widetilde \P^3 \ar[dashed]{r}& \P.
    \end{tikzcd}
  \]
  By \cite[Proposition~1.2]{alu.fab:93*1}, the map \(\widetilde \P^3 \dashrightarrow \P\) extends to a regular map \(\widetilde \P^3 \to \P\).
  By \Cref{lem:zmt}, we conclude the same about \(\widetilde \P^3 \dashrightarrow \Gr.\)
  (In \cite{alu.fab:93*1}, the base field is assumed to have characteristic 0, but Proposition~1.2 and its proof hold in any characteristic.)
\end{proof}

Let \(\Delta \subset \P^3 = \P \Hom(V, W)\) be the locus of degenerate maps.
Then \(\Delta\) is a quadric hypersurface, defined by the vanishing of the determinant.
We have an isomorphism
\[\Delta \xrightarrow{\sim} \P V \times \P W\]
given by
\begin{equation}\label{eq:delta}
  m \mapsto (\ker m, \im m).
\end{equation}
Under this isomorphism, the line \(L_b\) defined in \eqref{eq:lb} corresponds to the line \(\{b\} \times \P W\).

We now do a simple calculation that allows us to understand the map
\[\widetilde \phi \colon \widetilde \P \to \Gr(r+1, \Sym^d W).\]
Choose bases \(\langle  v_1, v_2 \rangle\) for \(V\) and \(\langle w_1, w_2 \rangle\) for \(W\).
Let us write elements of \(\Hom(V, W)\) as \(2 \times 2\)-matrices with respect to these bases.
Consider the one-parameter family of homomorphisms given by
\begin{equation}\label{eq:mt}
  m_t = \begin{pmatrix} 1 & 0 \\ 0 & t \end{pmatrix}, \quad t \in \A^1.
\end{equation}
For \(t \neq 0\), let
\[S_t = m_t(S) \subset \Sym^d W,\]
and
\[
  S_0 = \lim_{t \to 0} S_t \subset \Sym^d W.
\]
The following proposition describes \(S_0\).
\begin{proposition}\label{prop:key}
  Let \(S_0\) be defined as above.
  Let \(b \in \P V\) be the point corresponding to \([v_2] \subset V\) and let \(a_0 < \dots <a_r\) be the vanishing sequence of \(S\) at \(b\).
  Then \(S_0 \subset \Sym^d W\) is given by
  \[ S_0 = \langle  w_1^{d-a_0}w_2^{a_0}, \dots, w_1^{d-a_r}w_2^{a_r}\rangle.\]
\end{proposition}
\begin{proof}
  Let \(s_i \in S\) vanish to order \(a_i\) at \(b\).
  Then there exists a non-zero \(c_i \in \k\) and a homogeneous polynomial \(F_i(v_1,v_2)\) of degree \(d-a_i-1\) such that
  \[ s_i = c_i v_1^{d-a_i}v_2^{a_i} + v_2^{a_i+1}F_i(v_1,v_2). \]
  We see that
  \[ m_t(s_i) = c_i t^{a_i}w_1^{d-a_i}w_2^{a_i} + t^{a_i+1}w_2^{a_i+1}F_i(w_1,tw_2). \]
  Therefore, we get
  \[ \lim_{t \to 0} \langle m_t(s_i) \rangle = \langle  w_1^{d-a_i}w_2^{a_i}\rangle,\]
  and hence
  \[ \lim_{t \to 0} m_t(S) = \langle  w_1^{d-a_0}w_2^{a_0}, \dots, w_1^{d-a_r}w_2^{a_r}\rangle.\]
\end{proof}
\begin{corollary}\label{cor:gen}
  If \(b \in \P V\) is not a ramification point of \(S\), then \(S_0\) depends only on the subspace \(\im m_0 \subset W\).
\end{corollary}
\begin{proof}
  In this case, the vanishing sequence at \(b\) is \(0 < \dots < r\).
  From \Cref{prop:key}, we see that \(S_0 \subset \Sym^d W\) is the subspace of sections that vanish to order at least \(d-r\) at \(\im m_0 = [w_1] \in \P W\).
\end{proof}
\Cref{cor:gen} implies that the map \(\widetilde \phi\) is constant on the proper transforms of the lines \(\P V \times \{w\} \subset \Delta\) (see \Cref{fig:blowup}).

\begin{theorem}\label{thm:orbitclosure}
  Let \(S' \subset \Sym^d W\) be the image of \(S \subset \Sym^d V\) under an isomorphism \(V \to W\).
  The \(\GL W\)-orbit of \([S'] \in \Gr(r+1, \Sym^d W)\) contains finitely many orbits in its closure.
  In addition to the orbit of \([S']\), these are the orbits of \(\langle  w_1^{d-a_0}w_2^{a_0}, \dots, w_1^{d-a_r}w_2^{a_r}\rangle\) where \(a_0< \dots< a_r\) is the vanishing sequence of \(S\) at some \(p \in \P V\).
\end{theorem}
\begin{proof}
  The closure of the \(\GL W\)-orbit of \([S']\) is precisely the image of the morphism
  \[ \widetilde \phi \colon \widetilde \P \to \Gr(r+1, \Sym^d W).\]
  Let \(\widetilde \Delta \subset \widetilde \P\) be the pre-image of \(\Delta \subset \P^3\).
  Then \(\widetilde \P - \widetilde \Delta\) is a copy of \(\PGL_3\), and its image is the orbit of \([S']\).

  Consider \(\widetilde m \in \widetilde \Delta\) and let \(m \in \Delta\) be its image.
  Suppose \(m\) does not lie on any line \(L_b\) and let \(\im m = [w_1] \in \P W\).
  Then by \Cref{cor:gen}, the image \(\widetilde \phi(\widetilde m) \in \Gr(r+1, \Sym^d W)\) corresponds to
  \begin{equation}\label{eq:imgen}
    \langle w_1^d w_2^{0}, \dots, w_1^{d-r}w_2^{r}\rangle,
  \end{equation}
  where \(w_2 \in W\) is any vector linearly independent from \(w_1\).

  Suppose \(m\) lies on \(L_b\).
  Let \(w = \im m \in \P W\), so that \(m \in \Delta\) corresponds to \((b,w)\) in the isomorphism \eqref{eq:delta}.
  Suppose \(v_2 \in V\) is such that \(b = [v_2]\) and \(w_1 \in W\) such that \(w = [w_1]\).
  The point \(\widetilde m\) corresponds to a direction in the normal bundle \(N_{L_b/\P^3}\) at \(p\).
  Consider the direction in \(N_{L_b/\P^3}|_p\) corresponding to the line \(\P V \times \{w\} \subset \Delta\) through \(m\).
  By \Cref{cor:gen}, the map \(\widetilde \phi\) is constant on this line with value given by the subspace \eqref{eq:imgen}.
  If \(\widetilde m\) corresponds to this direction, then \(\widetilde \phi(\widetilde m)\) corresponds to the same subspace.
  All the other directions in \(N_{L_b/\P^3}|_p\) can be realised by tangent directions of families \(m_t\) for suitable choices of \(v_1 \in V\) and \(w_2 \in W\).
  If \(\widetilde m\) corresponds to the tangent direction given by \(m_t\), then \Cref{prop:key} implies that \(\widetilde \phi(\widetilde m)\) corresponds to the subspace
  \begin{equation}\label{eq:imspe}
    \langle  w_1^{d-a_0}w_2^{a_0}, \dots, w_1^{d-a_r}w_2^{a_r}\rangle.
  \end{equation}

  In conclusion, we see that the image of every point of \(\widetilde \Delta\) is of the form \eqref{eq:imgen}, which corresponds to the generic vanishing sequence, or \eqref{eq:imspe}, which corresponds to the vanishing sequence of a ramification point.
\end{proof}

\section{Localisation}\label{sec:localisation}
The goal of this section is to compute the equivariant orbit class of a linear series using localisation on \(\widetilde{\P}^{3}\).
The calculation is similar to the one in \cite[Appendix~B]{lee.pat.tse:19}.

Choose a basis \(\langle  w_{1}, w_{2} \rangle\) of \(W\) and let \(T \subset \GL(W)\) be the diagonal torus with respect to this basis.
Choose coordinates \(t_{1}, t_{2}\) on \(T\) so that the action on \(W\) is
\[ (t_{1},t_{2}) \colon (w_{1}, w_{2}) \mapsto (t_{1}w_{1}, t_{2}w_{2}).\]
Then the \(T\)-fixed locus of \(\P^{3} = \P \Hom(V,W)\) consists of two skew lines \(\Lambda_{1}\) and \(\Lambda_{2}\) given by
\[ \Lambda_i = \P \Hom(V, \langle w_{i}\rangle).\]
Note that the lines \(\Lambda_{i} \subset \P^{3}\) lie in the discriminant quadric \(\Delta\), and are in the opposite ruling as the lines \(L_{b}\) blown up to get \(\widetilde{\P}^{3}\).
The \(T\)-fixed locus of \(\widetilde{\P}^{3}\) consists of the proper transforms \(\widetilde{\Lambda}_{i}\) of \(\Lambda_{i}\) and points \(p_{i,b}\) for \(i = 1,2\) and \(b \in B\).
The point \(p_{i,b}\) maps under the blow-down to the point \(\Lambda_{i} \cap L_{b}\).
We now describe it in local coordinates.

Choose a basis \(v_{1}, v_{2}\) of \(V\) such that \([v_{2}] = b \in \P V\).
In the bases \(v_{1}, v_{2}\) of \(V\) and \(w_{1}, w_{2}\) of \(W\), we represent \(m \in \Hom(V, W)\) by the matrix
\[
  m = \begin{pmatrix}
    X_{11} & X_{12} \\
    X_{21} & X_{22}
  \end{pmatrix}.
\]
The matrix entries \(X_{ij}\) are functionals on \(\Hom(V, W)\).
The torus \(T\)-acts by
\[ (t_{1},t_{2}) \colon X_{ij} \mapsto t_{i}^{-1}X_{ij}.\]
The line \(\Lambda_{1} \subset \P^{3}\) is defined by \(X_{21} = X_{22} = 0\) and the line \(L_b \subset \P^3\) by \(X_{12} = X_{22} = 0.\)
Their intersection is the point \(X_{12} = X_{21} = X_{22} = 0\).
Choose affine coordinates \(x_{ij} = X_{ij}/X_{11}\) around this point.
Then the blow-up of \(\P^3\) in \(L_{b}\) has the coordinate charts
\[ (x_{21},x_{12},x_{22},u \mid x_{12}=ux_{22}) \text{ and } (x_{21},x_{12},x_{22}, v \mid x_{22}=vx_{12}).\]
The point \(p_{1,b}\) is contained in the second chart, and is given by
\begin{equation}\label{eq:p}
  p_{1,b}: x_{21} = x_{22} = v = 0.
\end{equation}

Given \((m_{1},m_{2}) \in \mathbf{Z}\), let \(\chi(m_1,m_2)\) be the character of \(T\) given by
\[ \chi(m_1,m_2) \colon (t_1,t_2) \mapsto t_{1}^{m_1} \cdot t_2^{m_2}.\]
The next two propositions give the normal bundles of the \(T\)-fixed loci.

\begin{proposition}\label{prop:normalp}
  Let \(N_{1,b}\) be the tangent space to \(p_{1,b} \in \widetilde{\P}^3\).
  As a \(T\)-representation, we have
  \[ N_{1,b} \cong \chi(-1,1) \oplus \chi(-1,1) \oplus \chi(1,-1). \]
\end{proposition}
\begin{proof}
  From \eqref{eq:p}, we see that the co-tangent space at \(p_{1,b}\) is generated by \(x_{21}, x_{22}, v\), on which \(T\) acts by
  \[ (t_1,t_2) \colon x_{21} \mapsto t_{1}t_2^{-1} x_{21} \quad x_{22} \mapsto t_1t_{2}^{-1}x_{22} \quad v \mapsto t_1^{-1}t_2v.\]
We now take the dual.
\end{proof}

\begin{proposition}\label{prop:normall}
  Let \(\widetilde N_{1}\) be the normal bundle of \(\widetilde \Lambda_1 \subset \widetilde{\P}^3\).
  In the \(T\)-equivariant Grothendieck group of \(\widetilde{\Lambda}_1 = \P^1\) (with the trivial \(T\)-action), we have
  \[ [N_1] = \left([O] + [O\left(2-|B|\right)]\right) \otimes \chi(-1,1).  \]
\end{proposition}
\begin{proof}
  Let \(N_{1}\) be the normal bundle of \(\Lambda_1 \subset \P^3\).
  The defining equations show that \(\Lambda_1\) is the zero locus of sections of \(O_{\P^3}(1) \otimes \chi(0,1)\) and the restriction of \(O_{\P^3}(1)\) to \(\Lambda_1\) is \(O(1) \otimes \chi(-1,0)\).
  Therefore, we get
  \[N_1 \cong \left( O(1) \oplus O(1) \right) \otimes \chi(-1,1).\]
  Recall that \(\beta \colon \widetilde \P^3 \to \P^3\) is the blow-up along \(\bigsqcup_b L_b\).
  We have the exact sequence
  \[ 0 \to \widetilde N_1 \to \beta^{*} N_1 \to Q \to 0,\]
  where \(Q \cong O_{\supp B} \otimes \chi(-1,1)\).
  The statement follows.
\end{proof}

Let \(\widetilde S\) be the pull-back to \(\widetilde \P^3\) of the universal sub-bundle of \(\Gr(r+1, \Sym^d W)\).
The following two propositions describe \(\widetilde S\) on the \(T\)-fixed locus.
\begin{proposition}\label{prop:sp}
  Let \(p = p_{1,b}\) and let \(a_0 < \dots < a_r\) be the vanishing sequence of \(S\) at \(b\).
  As a \(T\)-representation, we have
  \[ \widetilde S|_p \cong \bigoplus_{i=0}^{r} \chi(d-a_i, a_i).\]
\end{proposition}
\begin{proof}
  Consider the family \(m_t\) as in \eqref{eq:mt}, where \([v_2] = b\).
  Then, in \(\widetilde \P\) we have
  \[\lim_{t \to 0} m_t = p.\]
  Therefore, we have
  \[ \widetilde S|_p = \lim_{t \to 0} m_t(S).\]
  \Cref{prop:key} gives the limit and yields
  \[ \widetilde S|_p = \langle  w_1^{d-a_0}w_2^{a_0}, \dots, w_1^{d-a_r}w_2^{a_r}\rangle.\]
  The statement follows.
\end{proof}

\begin{proposition}\label{prop:sl}
  As a \(T\)-equivariant sheaf on \(\widetilde \Lambda_1\), we have
  \[ [\widetilde S]|_{\widetilde{\Lambda}_{1}} = \bigoplus_{i=0}^{r}\chi(d-i,i).\]
\end{proposition}
\begin{proof}
  By \Cref{cor:gen}, the map \(\widetilde \phi\) is constant on \(\widetilde \Lambda_1\), and its value corresponds to the subspace
  \[ \langle  w_1^dw_2^0, \dots, w_1^{d-r}w_2^r \rangle \subset \Sym^d W. \]
  The statement follows.
\end{proof}

We are now ready to prove the main theorem of the paper, which we restate.
\begin{theorem}\label{thm:realmain}
  Let \(s \in \Hom(\k^{r+1}, \Sym^d  \k^2)\) be an injective map with image \(S \subset \Sym^d\k^2\).
  Let \(\{(a_{i}(b)) \mid b \in B\}\) be the ramification profile of \(S\).
  Let \(\Gamma \subset  \PGL(2)\) be the stabiliser of \(S \subset \Sym^d\k^2\), and assume that it is finite.
  Let \(\Orb(s)\) be the closure of the \(\GL(r+1) \times \GL(2)\) orbit of \(s\).
  Then the class
  \[ |\Gamma| \cdot [\Orb(s)] \in A^{(r+1)(d-r)-3}_{\GL(r+1) \times \GL(2)}(\Hom(\k^{r+1}, \Sym^d \k^2)) \]
  is given by
  \[|\Gamma| \cdot [\Orb(s)] = \frac{1}{(\omega_1-\omega_2)^3}\sum_{b \in B}   \psi_{a_0(b), \dots, a_r(b)}(\mu;\omega) + \frac{2-|B|}{(\omega_1-\omega_2)^3} \cdot \psi_{0,\dots,r}(\mu;\omega).\]
\end{theorem}
\begin{proof}
  Set \(V = W = \k^2\) and consider the rational map
  \[ \P^3 = \P\Hom(V,W) \dashrightarrow \Gr(r+1, \Sym^d W)\]
  given by
  \[ m \mapsto m(S).\]
  Let \(\phi \colon \widetilde{\P}^3 \to \Gr(r+1, \Sym^dW)\) be its resolution provided by \Cref{thm:res}.
  Then \(\phi\) is a complete parametrisation of the \(\PGL W\)-orbit of \([S]\).
  By \Cref{prop:intgr}, we have
  \[ |\Gamma| \cdot [\Orb(s)] = \int_{\widetilde{\P}^3}c_{(r+1)(d-r)}\Hom(\underline \k^{r+1}, \phi^*Q),\]
  where \(Q\) is the universal quotient bundle on the Grassmannian.
  In our case, we have
  \[ \phi^* Q = \Sym^d \underline W / \widetilde{S}.\]
  Let \(T \subset \GL(2)\) be a maximal torus.
  By the equivariant localisation formula \cite[Corollary~1]{edi.gra:98*1}, we have
  \begin{equation}\label{eq:loc}
    \int_{\widetilde{\P}^3}c_{(r+1)(d-r)}\Hom(\underline \k^{r+1}, \phi^*Q) = \sum_F \int_F \frac{c_{(r+1)(d-r)}\Hom(\underline \k^{r+1}, \phi^*Q)}{e(N_F)},
  \end{equation}
  where the sum is over the \(T\)-fixed loci \(F\) and \(e(N_F)\) is the top Chern class of the equivariant normal bundle of \(F \subset \widetilde{\P}^3\).
  In our case, the fixed loci consist of the lines \(\widetilde \Lambda_i\) for \(i = 1, 2\) and the points \(p_{i,b}\) for \(i = 1,2\) and \(b \in B\).

  Let us compute the contribution from \(F = \widetilde\Lambda_1\).
  \Cref{prop:normall} give us the numerator of the integrand:
  \[
    \prod_{j = 0}^r \prod_{i = r+1}^d ((d-i)\omega_1 + i \omega_2 - \mu_j) = \phi_{0,\dots,r}(\mu;\omega).
  \]
  \Cref{prop:sl} gives us the denominator of the integrand:
  \[
    (\omega_2-\omega_1)(\omega_2-\omega_1+ (2-|B|)h),
  \]
  where \(h = c_1(O(1))\).
  The integral is the coefficient of \(h\) in the power series expansion of the quotient.
  We may compute it by formally differentiating with respect to \(h\) and setting \(h = 0\).
  The result is
  \begin{equation}\label{eq:intl}
    \int_{\widetilde{\Lambda}_1} \frac{c_{(r+1)(d-r)}\Hom(\k^{r+1},\phi^*Q)}{e(N_{\widetilde{\Lambda}_1})} =  \frac{2-|B|}{(\omega_1-\omega_2)^3} \cdot \phi_{0, \dots, r}(\mu;\omega).
  \end{equation}
  The contribution from \(F = \widetilde\Lambda_2\) is the same as above with \(\omega_1\) and \(\omega_2\) swapped.

Let us compute the contribution from \(F = p_{1,b}\).
\Cref{prop:sp} gives us the numerator:
\[ \prod_{j = 0}^{r} \prod_{i}((d-i)\omega_1 + i\omega_2 - \mu_j) = \phi_{a_0(b), \dots, a_r(b)}(\mu;\omega).\]
On the left hand side above, the index \(i\) ranges in the set \(\{0,\dots, d\} - \{a_0(b), \dots, a_r(b)\}\).
\Cref{prop:normalp} gives us the denominator:
\[ (\omega_1-\omega_2)^3.\]
Taking the integral is trivial in this case, so the result is
\begin{equation}\label{eq:intp}
\int_{p_{1,b}} \frac{c_{(r+1)(d-r)}\Hom(\k^{r+1},\phi^*Q)}{e(N_{p_{1,b}})} =  \frac{1}{(\omega_1-\omega_2)^3} \cdot \phi_{a_0(b), \dots, a_r(b)}(\mu;\omega).
\end{equation}
The contribution from \(F = p_{2,b}\) is the same as above with \(\omega_1\) and \(\omega_2\) swapped.

By combining \eqref{eq:intl}, \eqref{eq:intp}, and their analogues for \(\widetilde \Lambda_2\) and \(p_{2,b}\) we see that the integral in \eqref{eq:loc} is as claimed.
\end{proof}

\section{Orbit specialisation}\label{sec:spec}
The goal of this section is to understand limits of the orbit closure of a linear series under one-parameter specialisations.
Recall that \(\Orb\) denotes the orbit closure and \(\WOrb\) the weighted orbit closure, namely the cycle obtained by multiplying the orbit closure by the order of the stabiliser group in \(\PGL(2)\).

Fix a positive integer \(m\) and for each \(j = 1, \dots, m\), fix an increasing sequence of non-negative integers \(a_0(j)< \dots < a_r(j)\).
Assume that 
\begin{equation}\label{eqn:w}
  \sum_{j,i}(a_i(j) - i) = (r+1)(d-r).
\end{equation}
We consider one-parameter degenerations of linear series on \(\P^1\) with ramification profile \(\{(a_0(j), \dots, a_r(j)) \mid j = 1, \dots, m\}\).
We use limit linear series, for which we refer the reader to \cite{eis.har:86}.
Although \cite{eis.har:86} assumes characteristic 0, the theory holds without modification for characteristic \(> d\); see \cite{oss:06}.

Let \(\Delta\) be the spectrum of a DVR with special point \(0\) and generic point \(\eta\).
Let \(p_{1,\eta}, \dots, p_{m,\eta}\) be distinct sections \(\eta \to \P^1_\eta\) and let \((S_\eta, O(d))\) be a linear series that has vanishing sequence \(a_0(j)< \dots < a_r(j)\) at \(p_{j,\eta}\).
Let \((\mathcal X \to \Delta, p_1, \dots, p_m)\) be a semi-stable extension of \((\P^1_\eta \to \eta, p_{1,\eta}, \dots, p_{m,\eta})\) such that \(\mathcal X\) is non-singular.
Then the central fiber \((X, \overline p_{1}, \dots, \overline p_m)\) is a semi-stable pointed rational curve.
The linear series \(S_\eta\) yields a refined limit linear series \(\mathcal S\) on \(\mathcal X\) (use the implication \(3 \implies 1\) of \cite[Proposition~2.5]{eis.har:86}).
Let \(S\) be the reduction of \(\mathcal S\) to the central fiber.

We record the ramification profile of the limit linear series \(S\) using a decorated graph.
We begin with the dual graph \(\Gamma\) of the pointed curve \((X, \overline p_1, \dots, \overline p_m)\).
Recall that \(\Gamma\) is obtained by drawing a vertex \(\nu\) for every irreducible component \(X_\nu\) of \(X\), an edge \(\mu\) to \(\nu\) for every node of \(X\) on \(X_\mu\) and \(X_\nu\), and a half-edge \(i\) on \(\nu\) if the marked point \(\overline p_i\) lies on \(X_\nu\).
We think of the full-edge from \(\mu\) to \(\nu\) as a combination of two half-edges, one at \(\mu\) and one at \(\nu\).
We use the term \emph{dangling half-edge} when we refer to a true half-edge, and not a part of a full-edge.
On \(\Gamma\) we record the ramification information of \(S\) as follows.
We label a half-edge \(e\) incident to a vertex \(\nu\) by the vanishing sequence of the \(X_\nu\)-aspect of \(S\) at the point of \(X_\nu\) represented by \(e\).
If \(e\) is a full edge, joining \(\mu\) and \(\nu\) say, then it gets two labels, one at \(\mu\), say \((a_0, \dots, a_r)\), and one at \(\nu\), say \((b_0, \dots, b_r)\).
Since \(S\) is a refined series, the two labels are complementary; that is, they satisfy \(a_i + b_{r-i} = d\).
\Cref{fig:dualgraph} shows an example of a labelled dual graph \(\Gamma\) arising from a limit linear series of rank \(2\) and degree \(4\).

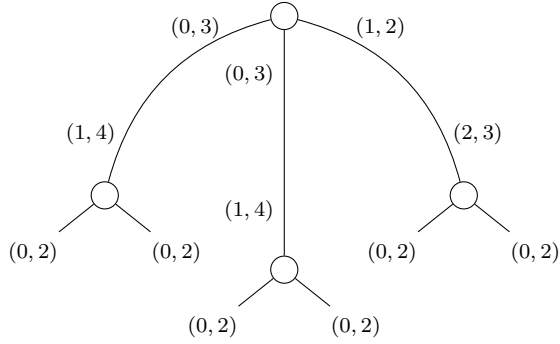
\begin{figure}
  \centering
  \begin{tikzpicture}
    \node[draw, circle] (A) {};
    \node[draw, circle] [below left=3cm of A] (B) {};
    \node[draw, circle] [below=3cm of A] (C) {};
    \node[draw, circle] [below right=3cm of A] (D) {};
    \node[below left=0.5cm of B] (B1) {\scriptsize \((0,2)\)};
    \node[below right=0.5cm of B] (B3) {\scriptsize \((0,2)\)};
    \node[below left=0.5cm of C] (C1) {\scriptsize \((0,2)\)};
    \node[below right=0.5cm of C] (C3) {\scriptsize \((0,2)\)};
    \node[below left=0.5cm of D] (D1) {\scriptsize \((0,2)\)};
    \node[below right=0.5cm of D] (D3) {\scriptsize \((0,2)\)};
        
    \draw (A) edge[bend right=30]
    node[left=.2cm, pos=0.1] {\scriptsize \((0,3)\)}
    node[left, pos=0.8] {\scriptsize \((1,4)\)} (B);
    \draw (A) edge
    node[left, pos=0.2] {\scriptsize \((0,3)\)} 
    node[left, pos=0.8] {\scriptsize \((1,4)\)} (C);
    \draw (A) edge[bend left=30]
    node[right=0.3cm, pos=0.1] {\scriptsize \((1,2)\)}
    node[right, pos=0.8] {\scriptsize \((2,3)\)} (D);

    \draw (B) edge (B1) edge (B3);
    \draw (C) edge (C1) edge (C3);
    \draw (D) edge (D1) edge (D3);    
  \end{tikzpicture}
    \caption{
    We represent the ramification information of a refined limit \(g^r_d\) by a labelled dual graph as above, where \(r = 1\) and \(d = 4\).
    The labels on the half-edges are vanishing sequences, and the two labels on a full-edge are complementary.
  }\label{fig:dualgraph}
\end{figure}

For a vertex \(\nu\) of \(\Gamma\), let \(S(\nu)\) be the \(X_\nu\)-aspect of the limit linear series \(S\).
\begin{theorem}\label{thm:degeneration}
  In the notation above, let \(\Sigma\) be the cycle on \(\Gr(r+1, \Sym^d\k^2)\) obtained as the limit of the cycle \(\WOrb(S_\eta)\).
  Then we have the equality of \(\GL(2)\)-invariant cycles 
  \[
    \Sigma = \sum_\nu \WOrb(S(\nu)).
  \]
\end{theorem}
\begin{proof}
  We first show that the cycle
  \[ D = \Sigma - \sum_{\nu \in \Gamma} \WOrb(S(\nu))\]
  is effective using \Cref{prop:manylimits}.

    For every vertex \(\nu\) of \(\Gamma\), let \(\mathcal X_\nu\) be the surface obtained by blowing down all components of \(X\) except \(X_\nu\).
  Then we have an isomorphism \(\mathcal X_\nu \cong \P^1 \times \Delta\); fix such an isomorphism and identify the two.
  Choose an arbitrary vertex of \(\Gamma\) and call it \(0\).
  For every \(\nu\), let \(g_\nu \in \PGL_2(\eta)\) be the unique element such that 
  \[
    \mathcal X_{0, \eta} = \P^1 \times \eta \xrightarrow{g_\nu} \P^1 \times \eta = \mathcal X_{\nu, \eta}\]
  represents the identity map
  \[\mathcal X_{0,\eta} = \mathcal X_\eta \to \mathcal X_{\nu, \eta} = \mathcal X_\eta.\]
    For \(\mu \neq \nu\), the identity map on the generic fiber does not extend to an isomorphism \(\mathcal X_\mu \to \mathcal X_\nu\).
  Therefore, \(g_\mu g_\nu^{-1} \in \PGL_2(\eta)\) does not extend to an element of \(\PGL_2(\Delta)\).
  On \(\mathcal X_\nu = \P^1 \times \Delta\), the limit of \([S_\eta] \from \eta \to \Gr(r+1,\Sym^d\k^2)\) is \([S(\nu)] \in \Gr(r+1,\Sym^d\k^2).\)
  Now, \Cref{prop:manylimits} applied to \(M = \Gr(r+1, \Sym^d\k^2)\) and \(G = \PGL_2\) implies that
  \[ D = \Sigma - \sum_\nu \WOrb(S(\nu))\]
  is effective.

  Let \(N = (d-r)(r+1)\).
  We apply the Wronskian map
  \[\rho \from \Gr(r+1, \Sym^d\k^2) \to \P\Sym^N\k^2\]
  and see that the cycle
  \[
    \rho_* D = \rho_* \Sigma - \sum_\nu \rho_* \WOrb(S(\nu))
  \]
  is effective.
  Let \(R_\eta\) be the ramification divisor of \(S_\eta\) and \(R(\nu)\) the ramification divisor of \(S(\nu)\).
  Then the cycle \(\rho_*\Sigma\) on \(\P \Sym^N \k^2\) is the limit of the cycle \(\WOrb(R_\eta)\) and \(\rho_*\WOrb(S(\nu))\) is the cycle \(\WOrb(R(\nu))\).
  Aluffi and Faber compute the degree of the weighted orbit closure for any divisor on \(\P^1\) (see \cite[Proposition~1.3]{alu.fab:93*1} and note that in their paper this degree is called the ``pre-degree'').
  Using their formula, it is straightforward to check that \(\deg \rho_* D = 0\) (see \Cref{rem:deg}).
  Since \(\rho_*D\) is effective, we conclude that \(\rho_*D = 0\) and since \(\rho\) is finite, that \(D = 0\).
\end{proof}
\begin{remark}\label{rem:deg}
  We indicate how to check that the pre-degree of the orbit of \(R_\eta\) is the sum of the pre-degrees of the orbits of \(R(\nu)\) using \cite[Proposition~1.3]{alu.fab:93*1}.
  The pre-degree of the orbit closure of a divisor depends only on the multiplicities of the points in the divisor.
  More precisely, by \cite[Proposition~1.3]{alu.fab:93*1}, the pre-degree for a divisor of degree \(N\) with multiplicities \(m_1, \dots, m_s\) is given by
  \begin{equation}\label{eq:ijk}
    \begin{split}
      p(m_1, \dots, m_s) &= N^3 - 3N \sum m_i^2 + 2 \sum m_i^3\\
      &= \sum_{i \neq j \neq k}m_im_jm_k.
    \end{split}
  \end{equation}
  To understand the multiplicities in the divisors \(R(\nu)\), consider the labelled dual graph \(\Gamma\) and replace the vanishing sequences by their weights, where the weight of the sequence \((a_0, \dots, a_r)\) is \(\sum_i (a_i-i)\).
  Then the multiplicities in \(R_\eta\) are the labels of the dangling half-edges and the multiplicities of \(R(\nu)\) are the labels of the half-edges incident to \(\nu\).
  The sum of the labels of all the half-edges incident to a vertex is \(N = (d-r)(r+1)\), and the sum of the two labels of a full-edge is also \(N\).
  Consider the operation of contracting an edge of \(\Gamma\) as shown below
  
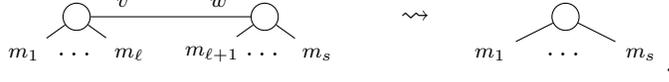
\[
    \begin{tikzpicture}
      \draw (0,0) node[draw, circle] (X) {};
      \draw (-1,-0.5) node (X1) {\scriptsize \(m_1\)} (0,-0.5) node (X2) {\dots} (1,-0.5) node (X3) {\scriptsize \(m_s\)};
      \draw (X) edge (X1) edge (X3);
  
      \begin{scope}[xshift=-6.5cm]
        \draw (0,0) node[draw, circle] (A) {} ; 
        \draw (-0.7,-0.5) node (A1) {\scriptsize \(m_1\)} (0,-0.5) node (A2) {\dots} (0.7,-0.5) node (A3) {\scriptsize \(m_\ell\)};
        \draw (A) edge (A1) edge (A3);
      \end{scope}
      \begin{scope}[xshift=-4cm]
        \draw (0,0) node[draw, circle] (B) {} ; 
        \draw (-0.7,-0.5) node (B1) {\scriptsize \(m_{\ell+1}\)} (0,-0.5) node (B2) {\dots} (0.7,-0.5) node (B3) {\scriptsize \(m_s\)};
        \draw (B) edge (B1) edge (B3);
      \end{scope}
      \draw (A) edge
      node[above, pos=0.2] {\scriptsize \(v\)}
      node[above, pos=0.8] {\scriptsize \(w\)}
      (B);
      \begin{scope}[xshift=-2cm]
        \node {\(\leadsto\)};
      \end{scope}
    \end{tikzpicture}.
  \]
  Since
  \[N = v+w = m_1 + \dots + m_\ell + v = m_{\ell+1} + \dots + m_s + w, \]
  we have \(v = m_{\ell+1} + \dots + m_s\) and \(w = m_1 + \dots + m_{\ell}\).
  From \eqref{eq:ijk} it follows that
  \[ p(m_1, \dots, m_\ell, v) + p(m_{\ell+1}, \dots, m_s, w) = p(m_1,\dots, m_s).\]
  By contracting the edges one by one, we reduce the sum of the pre-degrees for \(R(\nu)\) to a single sum, which is the pre-degree for \(R_\eta\).
\end{remark}
\begin{corollary}\label{thm:spec0}
  In the setup of \Cref{thm:degeneration}, in \(A^*_{\GL(2)}(\Gr(r+1,\Sym^d\k^2))\) we have
  \[ [\WOrb(S_\eta)] = \sum_\nu [\WOrb(S(\nu))].\]
\end{corollary}
\begin{proof}
  Since \([\WOrb(S_\eta)] = [\Sigma]\), the statement follows from \Cref{thm:degeneration}.
\end{proof}

Recall that the class \(\WOrb(S)\) depends only on the ramification profile of \(S\).
The ramification profile of \(S_\eta\) is recorded by the labels on the dangling half-edges of \(\Gamma\).
The ramification profile of \(S(\nu)\) is recorded by the labels on the half-edges incident to \(\nu\).
Thus, the equality in \Cref{thm:spec0} can be read off purely from the labelled graph \(\Gamma\).

A natural question is: which labelled graphs \(\Gamma\) arise from degenerations of linear series?
Here, the obvious necessary conditions are also sufficient.
Let \(\Gamma\) be a tree with \(m\) dangling half-edges.
Suppose all the half-edges of \(\Gamma\), including the full-edges considered as two half-edges, are labelled with increasing \((r+1)\)-tuples of integers in \(\{0,\dots,d\}\).
Assume that the two labels on a full-edge are complementary.
\begin{proposition}\label{prop:existence}
  Let \(\Gamma\) be a labelled graph as above.
  Suppose for every vertex \(\nu\) of \(\Gamma\), there is a linear series \(S(\nu)\) of rank \(r\) and degree \(d\) on \(\P^1\) whose ramification profile agrees with the multi-set of labels of half-edges incident to \(\nu\).
  Then \(\Gamma\) arises from a degeneration.
  That is, there exists a \(\Delta\), a semi-stable \(m\)-pointed curve \((\mathcal X \to \Delta, p_1, \dots, p_m)\) and a refined linear series \(\mathcal S\) on \(\mathcal X\) such that the labelled dual graph of its central fiber is \(\Gamma\).
\end{proposition}
\begin{proof}
  The proof is a simple application of the smoothing theorem for \(g^r_d\)'s \cite[Theorem~3.4]{eis.har:86}.
  
  Let \((X, \overline p_1, \dots, \overline p_m)\) be a semi-stable curve with dual graph \(\Gamma\).
  We assemble the linear series \(S(\nu)\) together to get a refined limit linear series \(S\) on \(X\).
  Note that the expected dimension of the space \(G^r_d\) on a rational curve with a prescribed ramification divisor is \(0\).
  Since the Wronskian map is finite and the ramification divisors of all aspects of \(S\) are prescribed, \(S\) is an isolated point of \(G^r_d(X, \overline p_1, \dots, \overline p_m)\).
  By \cite[Theorem~3.4]{eis.har:86}, there exists a \(\Delta\) and a smoothing \((\mathcal X \to \Delta, p_1, \dots, p_m, \mathcal S)\) of \((X, \overline p_1, \dots, \overline p_m, S)\).
\end{proof}

\begin{remark}\label{rem:existence}
  The existence of a \(g^r_d\) on \(\P^1\) with prescribed ramification profile is a Schubert theoretic condition on \(\Gr(r+1, \Sym^d\k^2)\).
  Such a series exists if and only if the product of certain Schubert classes on the Grassmannian is non-zero.
  See the remark after \cite[Theorem~2.3]{eis.har:83} for a precise statement.
\end{remark}

\Cref{thm:degeneration} asserts an equality of \(\GL(2)\)-invariant cycles on \(\Gr(r+1, \Sym^d\k^2)\).
It can be easily upgraded to an equality of \(\GL(r+1) \times \GL(2)\)-invariant cycles on \(\Hom(\k^{r+1}, \Sym^d \k^2)\).
\begin{proposition}\label{prop:degeneration2}
  In the notation of \Cref{thm:degeneration}, let \(s_\eta \colon \Hom(\k^{r+1}, \Sym^d \k^2)_\eta\) be a homomorphism with image \(S_\eta\).
  Let \(s_\nu \in \Hom(\k^{r+1}, \Sym^d\k^2)\) be homomorphisms with images \(S(\nu)\).
  Let \(\Sigma\) be the flat limit of \(\WOrb(s_\eta)\).
  Then, on \(\Hom(\k^{r+1}, \Sym^d\k^2)\), we have the equality of \(\GL(r+1) \times \GL(2)\)-invariant cycles
  \[ \Sigma = \sum_\nu \WOrb(s_\nu).\]
\end{proposition}
\begin{proof}
  Let \(U \subset \Hom(\k^{r+1}, \Sym^d\k^2)\) be the open subset of injective homomorphisms.
  The action of \(\GL(r+1)\) on \(U\) is free and the map
  \[\im \colon U \to \Gr(r+1, \Sym^d\k^2)\]
  taking a homomorphism to its image is the quotient by \(\GL(r+1)\).
  On \(U\), the assertions follow from \Cref{thm:degeneration}.

  It suffices to prove that \(\Sigma\) does contain any \((r+1)^2+3\)-dimensional components in the complement of \(U\).
  We claim that \(\Sigma\) is the closure of \(\Sigma \cap U\), from which the statement follows.

  To see this, take \(x \in \Sigma\).
  Then, possibly after a base change, we can choose \(x_\eta \in \Hom(\k^{r+1}, \Sym^d\k^2)_\eta\) in the \(\GL(r+1) \times \GL(2)\)-orbit of \(s_\eta\) such that the flat limit of \(x_\eta\) is \(x\).
  Let \(S \subset \Sym^d\k^2\) be the \((r+1)\)-dimensional subspace corresponding to the flat limit of \(\im(x_\eta)\) in \(\Gr(r+1, \Sym^d\k^2)\).
  Then it is easy to see that \(S\) contains the image of \(x\).
  Let \(\ell\) be the dimension of the image of \(x\).
  By making a change of coordinates on \(\k^{r+1}\) if necessary, assume that \(x\) sends the last \(r+1-\ell\) standard basis vectors to 0.
  Let \(s' \in \Hom(\k^{r+1}, \Sym^d\k^2)\) be a map that agrees with \(x\) on the first \(\ell\) basis vectors and whose image is \(S\).
  Then \(s' \in \Sigma \cap U\).
  Consider the diagonal one parameter subgroup \(\lambda_t \to \GL(r+1)\) with diagonal entries
  \[ (\underbrace{1, \dots, 1}_{\ell}, \underbrace{t, \dots, t}_{r+1-\ell}).\]
  Then, by construction, \(\lim_{t \to 0} \lambda_t(s') = x\).
  Since \(\Sigma \cap U\) is \(\GL(r+1)\)-invariant, and \(s' \in \Sigma \cap U\), we get that \(\lambda_t(s') \in \Sigma \cap U\) for \(t \neq 0\).
  We conclude that \(x\) lies in the closure of \(\Sigma \cap U\).  
\end{proof}
\begin{corollary}\label{thm:spec}
  In the setup of \Cref{prop:degeneration2}, in \(A^*_{\GL(r+1) \times \GL(2)}(\Hom(\k^{r+1},\Sym^d\k^2))\), we have
  \[ [\WOrb(s_\eta)] = \sum_\nu [\WOrb(s_\nu)].\]
\end{corollary}
\begin{proof}
  Since \([\WOrb(s_\eta)] = [\Sigma]\), the statement follows from \Cref{prop:degeneration2}.
\end{proof}

Finally, we discuss the specialisation mentioned in the introduction (\Cref{thm:introspec}).
Let \(p_1, \dots, p_m\) be sections \(\Delta \to \P^1 \times \Delta\) that are distinct on the generic fiber but \(p_1 = \dots = p_\ell\) on the central fiber (with no other coincidences).
Let \(S_\eta\) be a linear series on the generic fiber whose ramification points are \(p_1, \dots, p_m\) and whose vanishing sequence at \(p_i\) is \(a(i) = (a_0(i),\dots,a_r(i))\).
Let \(S\) be the flat limit of \(S_\eta\).
It is easy to see that the vanishing sequence of \(S\) at \(p_i\) for \(i > \ell\) is \(a(i)\).
Let \(c = (c_0, \dots, c_{r+1})\) be the vanishing sequence of \(S\) at \(p = p_1 = \dots = p_\ell\).
Let \(c' = (d-c_r, \dots, d-c_0)\) be the complementary sequence.
\begin{proposition}\label{prop:introspec}
  In the setup above, there exists a linear series of rank \(r\) and degree \(d\) on \(\P^1\) with ramification profile \(\{a(1), \dots, a(\ell), c'\}\).
  Furthermore, on \(\Hom(\k^{r+1}, \Sym^d\k^2)\) we have the equality of \(\GL(r+1) \times \GL(2)\)-equivariant classes
  \[ [\WOrb(a(1), \dots, a(m))] = [\WOrb(a(1), \dots, a(\ell), c')] + [\WOrb(a(\ell+1), \dots, a(m), c)].\]
\end{proposition}
\begin{proof}
  The existence of the linear series can be deduced using Schubert calculus.
  We give an alternate geometric proof using the smoothing theorem.

  Consider the pointed stable limit of \((\P^1_\eta, p_1, \dots, p_m)\).
  The dual graph of limit, together with the labelling given by the limit linear series, has the form
  
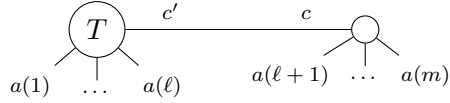
\[
    \begin{tikzpicture}
      \node[draw,circle] (A) {\(T\)};
      \node[draw,circle] (B)[right=3cm of A] {};
      \node[below left=0.3cm of A] (A1) {\scriptsize \(a(1)\)};
      \node [below=0.3cm of A] (A2) {\scriptsize \dots};
      \node [below right=0.3cm of A] (A3) {\scriptsize \(a(\ell)\)};
      \draw (A) edge (A1) edge (A2) edge (A3);
      \node[below left=0.3cm of B] (B1) {\scriptsize \(a(\ell+1)\)};
      \node [below=0.3cm of B] (B2) {\scriptsize \dots};
      \node [below right=0.3cm of B] (B3) {\scriptsize \(a(m)\)};
      \draw (B) edge (B1) edge (B2) edge (B3);
      \draw (A) edge
      node[above, pos=0.2] {\scriptsize \(c'\)}
      node[above, pos=0.8] {\scriptsize  \(c\)}
      (B);
    \end{tikzpicture}.
  \]
  In this diagram, \(T\) is a labelled tree with dangling half-edges labelled \(a(1), \dots, a(\ell)\), and one of its vertices, say \(t\), is connected to the remaining vertex, say \(o\), by an edge that has the label \(c'\) near \(t\) and \(c\) near \(o\).
  We apply \Cref{prop:existence} to \(T\) together with the dangling edges labelled \(a(1), \dots, a(\ell)\) and \(c'\).
  A geometric general fiber of the asserted degeneration gives a linear series of rank \(r\) and degree \(d\) on \(\P^1\) with ramification profile \(\{a(1), \dots, a(\ell), c'\}\).

  By applying \Cref{prop:existence} to the two vertex graph
  
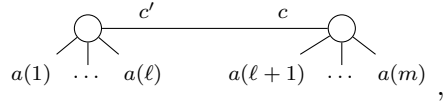
\[
    \begin{tikzpicture}
      \node[draw,circle] (A) {};
      \node[draw,circle] (B)[right=3cm of A] {};
      \node[below left=0.3cm of A] (A1) {\scriptsize \(a(1)\)};
      \node [below=0.3cm of A] (A2) {\scriptsize \dots};
      \node [below right=0.3cm of A] (A3) {\scriptsize \(a(\ell)\)};
      \draw (A) edge (A1) edge (A2) edge (A3);
      \node[below left=0.3cm of B] (B1) {\scriptsize \(a(\ell+1)\)};
      \node [below=0.3cm of B] (B2) {\scriptsize \dots};
      \node [below right=0.3cm of B] (B3) {\scriptsize \(a(m)\)};
      \draw (B) edge (B1) edge (B2) edge (B3);
      \draw (A) edge
      node[above, pos=0.2] {\scriptsize \(c'\)}
      node[above, pos=0.8] {\scriptsize  \(c\)}
      (B);
    \end{tikzpicture},
  \]
  we see that the graph above arises from a degeneration.
  We now apply \Cref{thm:spec} to this degeneration.  
\end{proof}

\bibliographystyle{abbrv}

\end{document}